\documentclass{amsart}
\usepackage{mathrsfs,amssymb,mathrsfs, multirow,setspace}
\usepackage{amsmath}
\usepackage[a4paper, total={7in, 9in}]{geometry}
\usepackage{enumerate}
\theoremstyle{plain}
\usepackage{graphicx}
\usepackage{mathtools}
\usepackage[utf8]{inputenc}

\newtheorem{theorem}{Theorem}[section]
\newtheorem{lemma}[theorem]{Lemma}
\newtheorem{proposition}[theorem]{Proposition}
\newtheorem{corollary}[theorem]{Corollary}

\theoremstyle{definition}
\newtheorem{notation}[theorem]{Notations}
\newtheorem{example}[theorem]{Example}

\theoremstyle{remark}
\newtheorem{remark}[theorem]{Remark}

\input xy
\xyoption{all}

\begin{document}
	\title[On the cozero-divisor graphs assosciated to rings ]{On the cozero-divisor graphs assosciated to rings}
\author[Praveen Mathil, Barkha Baloda Jitender Kumar]{Praveen Mathil, $\text{Barkha Baloda}^{^*}$, Jitender Kumar}
	\address{Department of Mathematics, Birla Institute of Technology and Science Pilani, Pilani, India}
	\email{maithilpraveen@gmail.com,barkha0026@gmail.com,jitenderarora09@gmail.com}

\begin{abstract}
Let $R$ be a ring with unity. The cozero-divisor graph of a ring $R$, denoted by
$\Gamma'(R)$, is an undirected simple graph whose vertices are the set of all
non-zero and non-unit elements of $R$, and two distinct vertices $x$ and $y$  are
adjacent if and only if $x \notin Ry$ and $y \notin Rx$. In this paper, first we study the Laplacian spectrum of $\Gamma'(\mathbb{Z}_n)$. We show that the graph $\Gamma'(\mathbb{Z}_{pq})$ is Laplacian integral. Further, we obtain the Laplacian spectrum of  $\Gamma'(\mathbb{Z}_n)$ for $n = p^{n_1}q^{n_2}$, where $n_1, n_2 \in \mathbb{N}$ and $p, q$ are distinct primes. In order to study the Laplacian spectral radius and algebraic connectivity of $\Gamma'(\mathbb{Z}_n)$, we characterized the values of $n$ for which the Laplacian spectral radius is equal to the order of $\Gamma'(\mathbb{Z}_n)$. Moreover, the values of $n$ for which the algebraic connectivity and vertex connectivity of $\Gamma'(\mathbb{Z}_n)$ coincide are also described. At the final part of this paper, we obtain the Wiener index of $\Gamma'(\mathbb{Z}_n)$ for arbitrary $n$.
\end{abstract}
\keywords{Cozero-divisor graph, ring of integers modulo $n$,  Laplacian spectrum, Wiener index \\ *  Corresponding author}
\maketitle

\section{Introduction}
The study of algebraic structures through graph theoretic properties has emerged as a fascinating research discipline in the past three decades, it has provided not only intriguing and exciting results but also opened up a whole new
domain yet to be explored. At the beginning, the idea to associate a graph with  ring structure was appeared in \cite{beck1988coloring}. The cozero-divisor graph related to a commutative ring was introduced by Afkhami \emph{et al.} in \cite{afkhami2011cozero}. The cozero-divisor graph of a ring $R$ with unity, denoted by
$\Gamma'(R)$, is an undirected simple graph whose vertex set is  the set of all
non-zero and non-unit elements of $R$, and two distinct vertices $x$ and $y$  are adjacent if and only if $x \notin Ry$ and $y \notin Rx$. They discussed certain basic properties on the structure of cozero-divisor graph and studied the relationship between the zero divisor graph and the cozero-divisor graph over ring structure. In \cite{afkhami2012cozero}, they investigated the complement of cozero-divisor graph and characterized the commutative rings with forest, star or unicyclic cozero-divisor graphs. Akbari \emph{et al.} \cite{akbari2014some}, studied the cozero-divisor graph associated to the polynomial ring. Some of the work associated with the cozero-divisor graph on the rings can be found in \cite{afkhami2012planar, afkhami2013cozero, akbari2014commutative, bakhtyiari2020coloring, mallika2017rings, nikandish2021metric}. The spectral graph theory is associated with spectral properties including investigation of charateristic polynomials, eigenvalues, eiegnvectors of matrices related with graphs. Recently, Chattopadhyay \emph{et al.} \cite{chattopadhyay2020laplacian} studied the Laplacian spectrum of the zero divisor graph of the ring $\mathbb{Z}_n$. They proved that the zero divisor graph of the ring $\mathbb{Z}_{p^l}$ is Laplacian integral for every prime $p$ and a positive integer $l \geq 2$. The work on spectral radius, viz. adjacency spectrum, Laplacian spectrum, signless Laplacian spectrum, distance signless spectrum etc.,  of the zero-divisor graphs can be found in \cite{chattopadhyay2020laplacian,magi2020spectrum, patil2021spectrum, pirzada2021signless, pirzada2020distance, pirzada2021normalized, rather2021laplacian}. The Wiener index, which is a distance based topological index, has various applications in pharmaceutical science, chemistry etc., see \cite{dobrynin2001wiener, janezic2015graph,wiener1947structural, xu2014survey}. Recently, the Wiener index of the zero divisor graph of the ring $\mathbb{Z}_n$ of integers modulo $n$ has been studied in \cite{asir2021wiener}.

In this paper, we study the Laplacian spectrum and the Wiener index of the cozero-divisor graph associated with the ring $\mathbb{Z}_n$. The paper is arranged as follows: In Section 2, we recall necessary results and fix our notations which are used throughout the paper. In Section 3, we study the structure of $\Gamma'(\mathbb{Z}_n)$. Section 4 deals with the Laplacian spectrum of the cozero-divisor graph of the ring $\mathbb{Z}_n$ for $n = p^{n_1}q^{n_2}$, where $p, q$ are distinct primes. In Section 5, the Laplacian spectral radius and the algebraic connectivity of $\Gamma'(\mathbb{Z}_n)$ have been investigated. The Wiener index of $\Gamma'(\mathbb{Z}_n)$ has been obtained in Section 6.

\section{preliminaries}
In this section, we recall  necessary definitions, results and notations of graph theory from \cite{westgraph}.
A graph $\Gamma$ is a pair  $ \Gamma = (V, E)$, where $V = V(\Gamma)$ and $E = E(\Gamma)$ are the set of vertices and edges of $\Gamma$, respectively. Let $\Gamma$ be a graph. The \emph{order} of a graph $\Gamma$ is the number of vertices of $\Gamma$. 
 Two distinct vertices $x, y \in \Gamma$ are $\mathit{adjacent}$, denoted by $x \sim y$, if there is an edge between $x$ and $y$. Otherwise, we denote it by $x \nsim y$.   
 The set $N_{\Gamma}(x)$ of all the vertices adjacent to $x$ in $\Gamma$ is said to be the \emph{neighbourhood} of $x$. 
 A \emph{subgraph} $\Gamma'$ of a graph $\Gamma$ is a graph such that $V(\Gamma') \subseteq V(\Gamma)$ and $E(\Gamma') \subseteq E(\Gamma)$. If $U \subseteq V(\Gamma)$ then the subgraph of $\Gamma$ induced by $U$, denoted by $\Gamma(U)$, is the graph with vertex set $U$ and two vertices of $\Gamma(U)$ are adjacent if and only if they are adjacent in $\Gamma$. The \emph{complement} $\overline{\Gamma}$ of $\Gamma$ is a graph with same vertex set as $\Gamma$ and distinct vertices $x, y$ are adjacent in $\overline{\Gamma}$ if they are not adjacent in $\Gamma$. A graph $\Gamma$ is said to be $complete$ if every two distinct vertices are adjacent. The complete graph on $n$ vertices is denoted by $K_n$. A path in a graph is a sequence of distinct vertices with the property that each vertex in the sequence is adjacent to the next vertex of it. The graph $\Gamma$ is said to be \emph{connected} if there is path between every pair of vertex.
 
 The distance between any two vertices $x$ and $y$ of $\Gamma$, denoted by $d(x,y)$, is the number of edges in a shortest path between $x$ and $y$. The Wiener index is defined as the sum of all distances between every pair of vertices in the graph that is the Wiener index of a graph $\Gamma$ is given by
\[ W(\Gamma) = \dfrac{1}{2}\sum_{u \in V(\Gamma)}\sum_{v \in V(\Gamma)} d(u,v) \]
The \emph{diameter} of a connected graph $\Gamma$, written as \rm{diam}$(\Gamma)$, is the maximum of the distances between vertices. If the graph consists of a single vertex, then the diameter is $0$. The \emph{degree} of a vertex $v \in \Gamma$, denoted by $\text{deg}(v)$, is the number of edges adjacent to $v$. The smallest degree among the vertices of $\Gamma$ is called the \emph{minimum degree} of $\Gamma$ and it is denoted by $\delta(\Gamma)$. A \emph{vertex cut-set} in a connected graph $\Gamma$ is a set $X$ of vertices such that the remaining subgraph $\Gamma \setminus X$, by removing the set $X$ is disconnected or has only one vertex. The \emph{vertex connectivity} of a connected graph $\Gamma$, denoted by $\kappa(\Gamma)$, is the minimum size of a vertex cut set. Let  $\Gamma_1$ and  $\Gamma_2$ be two graphs. The \emph{union} $\Gamma_1 \cup \Gamma_2$ is the graph with $V(\Gamma_1 \cup \Gamma_2) = V(\Gamma_1) \cup V(\Gamma_2)$ and $E(\Gamma_1 \cup \Gamma_2) = E(\Gamma_1) \cup E(\Gamma_2)$. The \emph{join} $\Gamma_1 \vee \Gamma_2$ of $\Gamma_1$ and $\Gamma_2$ is the graph obtained from the union of $\Gamma_1$ and $\Gamma_2$ by adding new edges from each vertex of $\Gamma_1$ to every vertex of $\Gamma_2$.

  Let $\Gamma$ be a graph on $k$ vertices and $V(\Gamma) = \{u_1, u_2, \cdots, u_k\}$. Suppose that $\Gamma_1, \Gamma_2, \cdots, \Gamma_k$ are $k$ pairwise disjoint graphs. Then the \emph{generalised join graph} $\Gamma[\Gamma_1, \Gamma_2, \cdots, \Gamma_k]$ of $\Gamma_1, \Gamma_2, \cdots, \Gamma_k$ is the graph formed by replacing each vertex $u_i$ of $\Gamma$ by $\Gamma_i$ and then joining each vertex of $\Gamma_i$ to every vertex of $\Gamma_j$ whenever $u_i \sim u_j$ in $\Gamma$ (cf. \cite{schwenk1974computing}).
 
For a finite simple (without multiple edge and loops) undirected graph $\Gamma$ with vertex set $V(\Gamma) = \{u_1, u_2, \ldots, u_k\}$, the \emph{adjacency matrix} $A(\Gamma)$ is defined as the $k\times k$ matrix whose $(i, j)th$ entry is $1$ if $u_i \sim u_j$ and $0$ otherwise. We denote the diagonal matrix by $D(\Gamma) = {\rm diag}(d_1, d_2, \ldots, d_k)$, where $d_i$ is the degree of the vertex $u_i$ of $\Gamma$. The \emph{Laplacian matrix} $\mathcal{L}(\Gamma)$ of $\Gamma$ is the matrix $D(\Gamma) - A(\Gamma)$. The matrix $\mathcal{L}(\Gamma)$ is  a symmetric and positive semidefinite, so that its eigenvalues are real and non-negative. Furthermore, the sum of each row (column) of $\mathcal{L}(\Gamma)$ is zero. 
 The eigenvalues of $\mathcal{L}(\Gamma)$ are called the \emph{Laplacian eigenvalues} of $\Gamma$ and are taken as  $\lambda_1(\Gamma) \geq \lambda_2(\Gamma) \geq \cdots \geq \lambda_n(\Gamma) = 0$. The second smallest Laplacian eigenvalue of $\mathcal{L}(\Gamma)$, denoted by $\mu(\Gamma)$, is called the \emph{algebraic connectivity} of $\Gamma$. The largest Laplacian eigenvalue $\lambda(\Gamma)$ of $\mathcal{L}(\Gamma)$ is called the \emph{Laplacian spectral radius} of $\Gamma$. Now let $\lambda_{1}(\Gamma) \geq \lambda_{2}(\Gamma) \geq \cdots \geq \lambda_{r}(\Gamma) = 0$ be the distinct eigenvalues of $\Gamma$ with multiplicities $\mu_1, \mu_2, \ldots, \mu_r$, respectively. The \emph{Laplacian spectrum} of $\Gamma$, that is the spectrum  of $\mathcal{L}(\Gamma)$, is represented as
 \begin{center}
$\displaystyle \Phi(\mathcal{L}(\Gamma)) = \begin{pmatrix}
\lambda_{1}(\Gamma) & \lambda_{2}(\Gamma) & \cdots& \lambda_{r}(\Gamma)\\
 \mu_1 & \mu_2 & \cdots & \mu_r
\end{pmatrix}$. 
\end{center}
Sometime we write $\Phi(\mathcal{L}(\Gamma)$  as $\Phi_{\mathcal{L}}(\Gamma)$ also. The following results are useful in the sequel.

\begin{theorem}\cite{cardoso2013spectra}\label{laplacianspectrum theoremforjoin}
Let $\Gamma$ be a graph on $k$ vertices having $V(\Gamma) = \{u_1, u_2, \cdots, u_k\}$ and let $\Gamma_1, \Gamma_2, \cdots, \Gamma_k$ be $k$ pairwise disjoint graphs on $n_1, n_2, \cdots, n_k$ vertices, respectively. Then the Laplacian spectrum of $\Gamma[\Gamma_1, \Gamma_2, \cdots, \Gamma_k]$ is given by
\begin{equation}\label{lapacianequation}
\Phi_{L}(\Gamma[\Gamma_1, \Gamma_2, \cdots, \Gamma_k]) = \bigcup\limits_{i=1}^{k} (D_{i} + (\Phi_{L}(\Gamma_i) \setminus \{0\})) \bigcup \Phi(\mathbb{L}(\Gamma))
    \end{equation}
where \[ D_i = \begin{cases} 
      \sum \limits_{u_j \sim u_i}n_j & ~~\text{if}~~ N_{\Gamma}(u_i) \neq \emptyset;\\
      0 & \rm{otherwise} 
    \end{cases}
\]
\begin{equation}
\mathbb{L}({\Gamma})  = \displaystyle \begin{bmatrix}
	D_{1}&  -p_{1,2} & \cdots & -p_{1,k}  \\
	-p_{2,1}& D_{2}   &\cdots&  -p_{2,k} \\ 
 \cdots & \cdots & \cdots & \cdots \\
 -p_{k,1} & -p_{k,2} & \cdots & D_{k}
	\end{bmatrix}
\end{equation}
such that \[ p_{i,j} = \begin{cases} 
      \sqrt{n_in_j} & ~~\text{if}~~ u_i \sim u_j~~ \text{in}~~ \Gamma\\
      0 & \rm{otherwise} 
    \end{cases}
\]
in (\ref{lapacianequation}), $(\Phi_{L}(\Gamma_i) \setminus \{0\}))$ means that one copy of the eigenvalue $0$ is removed from the multiset $\Phi_{L}(\Gamma_i)$, and $D_i+(\Phi_{L}(\Gamma_i) \setminus \{0\}))$ means $D_i$ is added to each element of $(\Phi_{L}(\Gamma_i) \setminus \{0\}))$.\\

Let $\Gamma$ be a weighted graph by assigning the weight $n_i = |V(\Gamma_i)|$ to the vertex $u_i$ of $\Gamma$ and $i$ varies from $1$ to $k$. Consider $L(\Gamma)= (l_{i,j})$ to be a $k \times k$ matrix, where 
\[ l_{i,j} = \begin{cases}
      -n_j &~~~~\text{if} ~~i \neq j ~~\text{and}~~u_i \sim u_j;\\
      \sum \limits_{u_i \sim u_r}n_r & ~~\text{if}~~ i=j;\\
      0 & \rm{otherwise.} 
    \end{cases}
\]
\end{theorem}

The matrix $L(\Gamma)$ is called the vertex weighted Laplacian matrix of $\Gamma$, which is a zero row sum matrix but not a symmetric matrix in general.
Though the $k \times k$ matrix $\mathbb{L}(\Gamma)$ defined in Theorem \ref{laplacianspectrum theoremforjoin}, is a symmetric matrix but it need not be a zero row sum matrix. Since the matrices $\mathbb{L}(\Gamma)$  and $L(\Gamma)$ are similar, we have the following remark.
\begin{remark}\label{laplacianremark}
$\Phi(\mathbb{L}(\Gamma)) = \Phi(L(\Gamma))$.
\end{remark}
Let $\mathbb{Z}_n$ denotes the ring of integers modulo $n$ that is, $\mathbb{Z}_n = \{0, 1, \cdots, n-1\}$. The number of integers which are prime to $n$ and less than $n$ is denoted by \emph{Euler Totient function} $\phi(n)$. An integer $d$, where $1 < d < n$, is called a proper divisor of $n$ if $d|n$. If $d$ does not divide $n$ then we write it as $d \nmid n$. The number of all the divisors of $n$ is denoted by $\tau(n)$. The greatest common divisor of the two positive integers $a$ and $b$ is denoted by \text{gcd}$(a, b)$. The ideal generated by the element $a$ of $\mathbb{Z}_n$ is the set $\{xa \; : \; x \in \mathbb{Z}_n\}$ and it is denoted by $\langle a \rangle$. 


\section{Structure of the cozero-divisor graph $\Gamma'(\mathbb{Z}_n)$}
In this section, we discuss about the structure of the cozero-divisor graph $\Gamma'(\mathbb{Z}_n)$.  Let $d_1, d_2, \cdots, d_k$ be the proper divisors of $n$. For $1 \leq i \leq k$, consider the following sets
\begin{center}
    $\mathcal{A}_{d_{i}}$ = $\{ x \in \mathbb{Z}_{n}: \text{gcd}(x,n) = d_{i}\}$.
\end{center}

\begin{remark}\label{partition}
For $x, y \in \mathcal{A}_{d_{i}}$, we have $\langle x \rangle = \langle y \rangle = \langle d_i \rangle$. Further, note that the sets $\mathcal{A}_{d_{1}}, \mathcal{A}_{d_{2}}, \cdots, \mathcal{A}_{d_{k}}$ forms a partition of the vertex set of the graph $\Gamma'(\mathbb{Z}_n)$. Thus, $V(\Gamma'(\mathbb{Z}_n)) = \mathcal{A}_{d_{1}} \cup \mathcal{A}_{d_{2}} \cup \cdots \cup \mathcal{A}_{d_{k}}$. 
\end{remark}

The cardinality of each $\mathcal{A}_{d_i}$ is known in the following lemma.
\begin{lemma}\cite{young2015adjacency}\label{valueof partition}
$|\mathcal{A}_{d_{i}}| = \phi(\frac{n}{d_i})$ for $1 \leq i \leq k$.
\end{lemma}

\begin{lemma}\label{adjacenyofvertex}
Let $x \in \mathcal{A}_{d_{i}}$, $ y \in \mathcal{A}_{d_{j}}$, where $i, j \in \{1, 2, \cdots, \tau(n)-2\}$. Then $x \sim y$ in $\Gamma'(\mathbb{Z}_n)$ if and only if $d_i \nmid d_j$ and $d_j \nmid d_i$.
\end{lemma}
\begin{proof}
First note that in $\mathbb{Z}_n$, $x \in \langle y \rangle$ if and only if $y \mid x$. Let $x \in \mathcal{A}_{d_{i}}$ and $y \in \mathcal{A}_{d_{j}}$ be two distinct vertices of $\Gamma'(\mathbb{Z}_n)$. Suppose that $x \sim y$ in $\Gamma'(\mathbb{Z}_n)$. Then $x \notin \langle y \rangle$ and $y \notin \langle x \rangle$. If $d_i \mid d_j$, then $d_j \in \langle d_i \rangle = \langle x \rangle$. It follows that $\langle y \rangle = \langle d_j \rangle \subseteq \langle x \rangle$ and so $y \in \langle x \rangle$, which is not possible. Similarly, if $d_j \mid d_i$, then we get $x \in \langle y \rangle$, again a contradiction. Thus, neither $d_i \mid d_j$ nor $d_j \mid d_i$. Conversely, if $d_i \nmid d_j$ and $d_j \nmid d_i$ then we obtain $x \notin \langle y \rangle$ and $y \notin \langle x \rangle$. It follows that $x \sim y$. The result holds.
\end{proof}

 For distinct vertices $x, y$ of $\mathcal{A}_{d_{i}}$, by Remark \ref{partition}, clearly $x \in \langle y \rangle$ and $y \in \langle x \rangle$. It  follows that $x \nsim y$ in $\Gamma'(\mathbb{Z}_n)$. Using Lemma \ref{valueof partition}, we have the following corollary.  

\begin{corollary}\label{partitionofcozerodivisorgraphisomorphic}
The following statements hold:
\begin{itemize}
    \item[(i)] For $i \in \{1, 2, \cdots, \tau(n)-2\}$, the induced subgraph $\Gamma'(\mathcal{A}_{d_i})$ of $\Gamma'(\mathbb{Z}_n)$ is isomorphic to $\overline{K}_{\phi(\frac{n}{d_i})}$.
    \item [(ii)] For $i,j \in \{1, 2, \cdots, \tau(n)-2\}$ and $i \neq j$, a vertex of $\mathcal{A}_{d_i}$ is adjacent to either all or none of the vertices of $\mathcal{A}_{d_j}$. 
\end{itemize}
\end{corollary}

Thus, the partition $\mathcal{A}_{d_{1}}, \mathcal{A}_{d_{2}}, \ldots,  \mathcal{A}_{d_{\tau(n) - 2}}$ of $V(\Gamma'(\mathbb{Z}_n))$ is an equitable partition in such a way that every vertex of the $\mathcal{A}_{d_{i}}$ has equal number of neighbors in $\mathcal{A}_{d_{j}}$ for every $i, j \in \{1, 2, \cdots, \tau(n)- 2\}$.

We define $\Upsilon'_n$ by the simple undirected graph whose vertex set is the set of all proper divisors $d_1, d_2, \cdots, d_k$ of $n$ and two distinct vertices $d_i$ and $d_j$ are adjacent if and only if $d_i \nmid d_j$ and $d_j \nmid d_i$.

\begin{lemma}\label{connectednessofcozero}
For a prime $p$, the graph $\Upsilon'_n$ is connected if and only if $n \neq p^t$, where $t \geq 3$.
\end{lemma}
\begin{proof}
Suppose that $\Upsilon'_n$ is a connected graph and $V(\Upsilon'_n) = \{d_1, d_2, \cdots, d_k\}$. If $n = p^t$ for $t \geq 3$, then $V(\Upsilon'_{p^t}) = \{p, p^2, \cdots, p^{t-1}\}$. The definition of $\Upsilon'_n$ gives that $\Upsilon'_{p^t}$ is a null graph on $t-1$ vertices. Thus, $\Upsilon'_n$ is not connected; a contradiction. Conversely, suppose that $n \neq p^t$, where $t \geq 3$. If $n = p^t$ for $t \in \{1, 2\}$, then there is nothing to prove because $\Upsilon'_p$ is an empty graph whereas $\Upsilon'_{p^2}$ is a graph with one vertex only. We may now suppose that $n = p_1^{n_1}p_2^{n_2}\cdots p_m^{n_m}$, where $p_i$'s are distinct primes and $m \geq 2$. Now let $d, d' \in V(\Upsilon'_n)$. If $d \nmid d'$ and $d' \nmid d$, then $d \sim d'$. Without loss of generality, assume that $d \mid d'$ with $d = p_1^{\beta_1}p_2^{\beta_2}\cdots p_m^{\beta_m}$ and $d' = p_1^{\alpha_1}p_2^{\alpha_2}\cdots p_m^{\alpha_m}$. Note that $\alpha_i, \beta_i \in \mathbb{N} \cup \{0\}$ such that $\beta_i \leq \alpha_i$. Since $d'$ is a proper divisor of $n$ there exists $r \in \{1, 2, \ldots, m\}$, where $\alpha_r \leq n_r$,  such that $p_r^{n_r} \nmid d'$ and $d' \nmid p_r^{n_r}$. Clearly, $p_r^{n_r} \nmid d$. If $d \nmid p_r^{n_r}$, then $d \sim p_r^{n_r} \sim d'$. If $d \mid p_r^{n_r}$, then there exists $s \in \{1, 2, \ldots, m\}\setminus \{r\}$ such that $d \nmid p_s$ and $p_s \nmid d$. Also, $p_r^{n_r} \nmid p_s$ and $p_s \nmid p_r^{n_r}$. It follows that $d' \sim p_r^{n_r} \sim p_s \sim d$. Hence, the graph $\Upsilon'_n$ is connected.
\end{proof}

\begin{lemma}\label{inducedsubgraphequaltogamma}
$\Gamma'(\mathbb{Z}_n) = \Upsilon'_n [\Gamma'(\mathcal{A}_{d_1}), \Gamma'(\mathcal{A}_{d_2}), \cdots, \Gamma'(\mathcal{A}_{d_k})],$
where $d_1, d_2, \cdots, d_k$ are all the proper divisors of $n$. 
\end{lemma}

\begin{proof}
Replace the vertex $d_i$ of $\Upsilon'_n$ by $\Gamma'(\mathcal{A}_{d_i})$ for $1 \leq i \leq k$. Consequently, the result can be obtained by using Lemma \ref{adjacenyofvertex}. 
\end{proof}

\begin{lemma}
For a prime $p$, we have $\Gamma'(\mathbb{Z}_n)$ is connected if and only if either $n = 4$ or $n \neq p^t$, where $t \geq 2$.
\end{lemma}
\begin{proof}
Suppose that $\Gamma'(\mathbb{Z}_n)$ is a connected graph and $n \ne 4$. If possible, let $n = p^t$ for $t \geq 2$ then note that $V(\Gamma'(\mathbb{Z}_n)) = \Gamma'(\mathcal{A}_{p}) \cup \Gamma'(\mathcal{A}_{p^2}) \cup \cdots \cup \Gamma'(\mathcal{A}_{p^{t-1}})$ and so $x \nsim y$ for any $x, y \in V(\Gamma'(\mathbb{Z}_n))$ (see Lemma \ref{adjacenyofvertex} and Corollary \ref{partitionofcozerodivisorgraphisomorphic}). Consequently, $\Gamma'(\mathbb{Z}_n)$ is a null graph; a contradiction. Thus, $n \neq p^t$, where $t \geq 2$. Converse follows by the proof of Lemma \ref{connectednessofcozero} and  Lemma \ref{inducedsubgraphequaltogamma}. 
\end{proof}

\begin{example}
The cozero-divisor graph $\Gamma'(\mathbb{Z}_{30})$ is shown in Figure 2. 
\begin{figure}[h!]
\centering
\includegraphics[width=0.4 \textwidth]{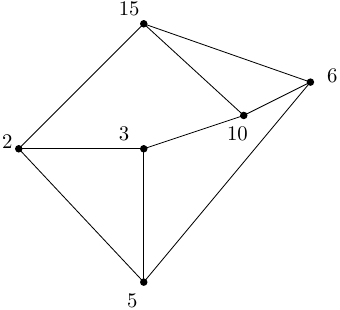}
\caption{The graph $\Upsilon'_{30}$}
\end{figure} \\
By Lemma \ref{inducedsubgraphequaltogamma}, note that  $\Gamma'(\mathbb{Z}_{30}) = \Upsilon'_{30} [\Gamma'(\mathcal{A}_{2}), \Gamma'(\mathcal{A}_{3}), \Gamma'(\mathcal{A}_{5}), \Gamma'(\mathcal{A}_{6}), \Gamma'(\mathcal{A}_{10}),  \Gamma'(\mathcal{A}_{15})]$, where $\Upsilon'_{30}$ is shown in Figure 1 and  $\Gamma'(\mathcal{A}_{2}) = \overline{K}_8$, $\Gamma'(\mathcal{A}_{3}) = \overline{K}_4 = \Gamma'(\mathcal{A}_{6})$, $\Gamma'(\mathcal{A}_{5}) = \overline{K}_2 = \Gamma'(\mathcal{A}_{10})$,  $\Gamma'(\mathcal{A}_{15}) = \overline{K}_1$. 
\begin{figure}[h!]
\centering
\includegraphics[width=0.4 \textwidth]{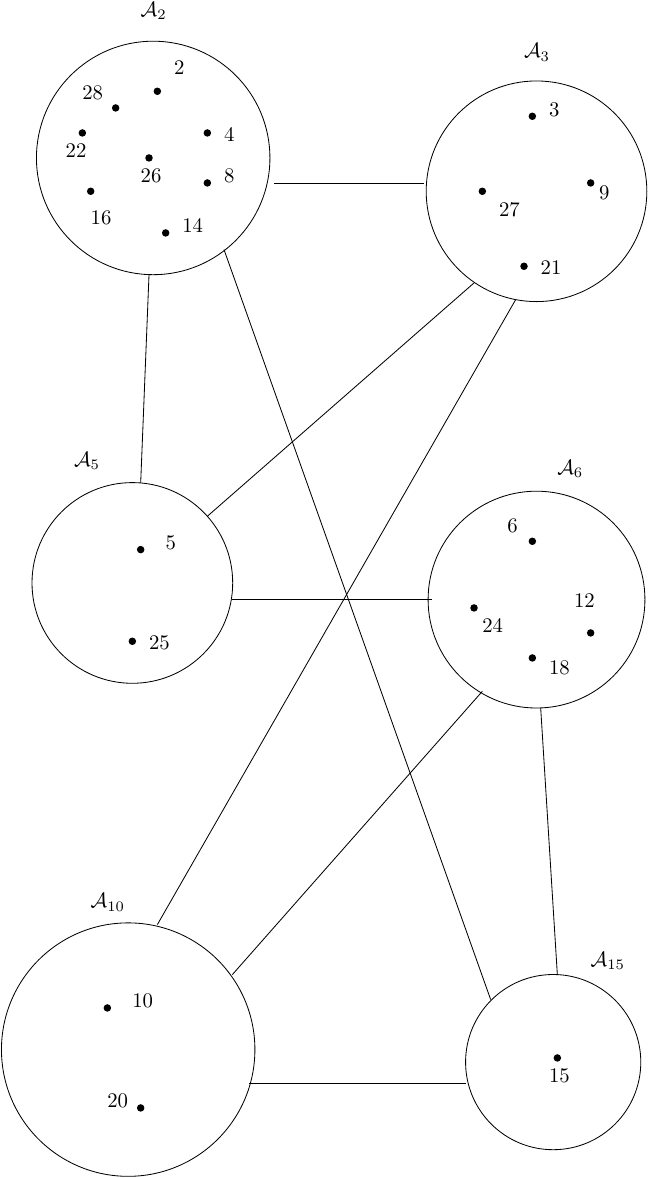}
\caption{The graph $\Gamma'(\mathbb{Z}_{30})$}
\end{figure}
\end{example}

\section{Laplacian spectrum of $\Gamma'(\mathbb{Z}_n)$}
In this section, we investigate the Laplacian spectrum of the  $\Gamma'(\mathbb{Z}_n)$ for various $n$. Consider $d_1, d_2, \cdots,d_k$ as all the proper divisors of $n$. For $1 \leq i \leq k$, we give the weight $\phi(\frac{n}{d_i}) = |\mathcal{A}_{d_{i}}|$ to the vertex $d_i$ of the graph $\Upsilon'_n$. Define the integer
\begin{center}
    $D_{d_j} =  \sum \limits_{{d_i} \in N_{\Upsilon'_n}(d_j)} \phi(\frac{n}{d_i})$
\end{center}
The $k \times k$ weighted Laplacian matrix $L(\Upsilon'_{n})$ of $\Upsilon'_{n}$ defined in Theorem \ref{laplacianspectrum theoremforjoin} is given by 
\begin{equation}
L(\Upsilon'_{n})  = \displaystyle \begin{bmatrix}\label{spectrum matrix2}
	D_{d_1}&  -l_{1,2} & \cdots & -l_{1,k}  \\
	-l_{2,1}& D_{d_2}   &\cdots&  -l_{2,k} \\ 
 \cdots & \cdots & \cdots & \cdots \\
 -l_{k,1} & -l_{k,2} & \cdots & D_{d_k}
	\end{bmatrix}
\end{equation}
where \[ l_{i,j} = \begin{cases} 
      \phi(\frac{n}{d_j}) & ~~\text{if}~~ d_i \sim d_j~~ \text{in}~~ \Upsilon'_{n};\\
      0 & \rm{otherwise.} 
    \end{cases}
\]

\begin{theorem}\label{laplacianeigenvaluescozero}
The Laplacian spectrum of $\Gamma'(\mathbb{Z}_n)$ is given by \begin{center}
    $\Phi_{L}(\Gamma'(\mathbb{Z}_n)) = \bigcup\limits_{i=1}^{k} (D_{d_i} + (\Phi_{L}(\Gamma'(\mathcal{A}_{d_i})) \setminus \{0\})) \bigcup \Phi(L(\Upsilon'_{n}))$,
\end{center}
where $D_{d_i} + (\Phi_{L}(\Gamma'(\mathcal{A}_{d_i})) \setminus \{0\})$ represents that $D_{d_i}$ is added to each element of the multiset $(\Phi_{L}(\Gamma'(\mathcal{A}_{d_i})) \setminus \{0\})$.
\end{theorem}

\begin{proof}
By Lemma \ref{inducedsubgraphequaltogamma}, $\Gamma'(\mathbb{Z}_n) = \Upsilon'_n [\Gamma'(\mathcal{A}_{d_1}), \Gamma'(\mathcal{A}_{d_2}), \cdots, \Gamma'(\mathcal{A}_{d_k})]$. Consequently, by Theorem \ref{laplacianspectrum theoremforjoin} and Remark \ref{laplacianremark}, the result holds.
\end{proof}
If $n = p^t$, where $t >1$, then the graph $\Gamma'(\mathbb{Z}_n)$ is a null graph. Let $n \neq p^t$ for any $t \in \mathbb{N}$. Then  by Lemma \ref{connectednessofcozero}, $\Upsilon'_n$ is connected graph so that $D_{d_i} > 0$. By Theorem \ref{laplacianeigenvaluescozero}, out of $n - \phi(n) -1$ Laplacian eigenvalues of $\Gamma'(\mathbb{Z}_n)$ note that $n - \phi(n) -1-k$ eigenvalues are non-zero integers. The remaining $k$ Laplacian eigenvalues of $\Gamma'(\mathbb{Z}_n)$ are the roots of the characteristic equation of the matrix $L(\Upsilon'_{n})$ given in equation (\ref{spectrum matrix2}).

\begin{lemma}
Let $n = pq$ be a product of two distinct primes. Then the Laplacian spectrum of $\Gamma'(\mathbb{Z}_n)$ is given by
\[\displaystyle \begin{pmatrix}
0 & p+q-2 &  p-1 & q-1 \\
 1 & 1 & q-2 & p-2 \\
\end{pmatrix}.\]
\end{lemma}

\begin{proof}
By Lemma \ref{inducedsubgraphequaltogamma}, we have $\Gamma'(\mathbb{Z}_{pq}) = \Upsilon'_{pq}[\Gamma'(\mathcal{A}_{p}), \Gamma'(\mathcal{A}_{q})]$, where $\Upsilon'_{pq} = K_2$, $\Gamma'(\mathcal{A}_{p}) = \overline{K}_{\phi(q)}$ and $\Gamma'(\mathcal{A}_{q}) = \overline{K}_{\phi(p)}$ (cf. Lemma \ref{valueof partition} and Corollary \ref{partitionofcozerodivisorgraphisomorphic}). Consequently, by Theorem \ref{laplacianeigenvaluescozero}, the Laplacian spectrum of $\Gamma'(\mathbb{Z}_{pq})$ is 

\begin{align*}
\Phi_{L}(\Gamma'(\mathbb{Z}_{pq})) &= (D_{p} + (\Phi_{L}(\Gamma'(\mathcal{A}_{p})) \setminus \{0\})) \bigcup (D_{q} + (\Phi_{L}(\Gamma'(\mathcal{A}_{q})) \setminus \{0\})) \bigcup \Phi(L(\Upsilon'_{pq}))\\
&= \displaystyle \begin{pmatrix}
 p-1 & q-1 \\
 q-2 & p-2 \\
\end{pmatrix} \bigcup \Phi(L(\Upsilon'_{pq})).
\end{align*}

Then the matrix \[L(\Upsilon'_{pq})  = \displaystyle \begin{bmatrix}
	p-1 &  -(p-1) \\
	-(q-1)& q-1   
\end{bmatrix}\]
 has eigenvalues $p+q-2$ and $0$. Thus, we have the result. 
\end{proof}

\begin{notation}
$(\lambda_{i})^{[\mu_{i}]}$ denotes the eigenvalue $\lambda_{i}$ of $\mathcal{L}({\Gamma'(\mathbb{Z}_n)})$ with multiplicity $\mu_{i}$. 
\end{notation}

\begin{lemma}
For distinct primes $p$ and $q$, if $n = p^2q$ then the Laplacian eigenvalues of $\Gamma'(\mathbb{Z}_n)$ consists of the set $\left\{(p^2-p)^{[(p-1)(q-1)-1]},(pq-p)^{[p^2-p-1]}, (p^2-1)^{[q-2]}, (q-1)^{[p-2]} \right\}$ and the remaining eigenvalues are the roots of 
the characteristic polynomial
\[x^4-\{(p-1)(2p+1)+(p+1)(q-1)\}x^3+\{p(p-1)^2(p+1)+(p-1)(p+1)^2(q-1)+p(q-1)^2+(p-1)^2(q-1)\}x^2\]
\[\hspace{-11cm}-p(p-1)(q-1)\{(p-1)(p+1)+p(q-1)\}x.\]
\end{lemma}
\begin{proof}
First note that $\Upsilon'_{p^2q}$ is the path graph given by $p \sim q \sim p^2 \sim pq$. By Lemma \ref{inducedsubgraphequaltogamma}, \[\Gamma'(\mathbb{Z}_{p^2q}) = \Upsilon'_{p^2q}[\Gamma'(\mathcal{A}_{p}), \Gamma'(\mathcal{A}_{q}), \Gamma'(\mathcal{A}_{p^2}), \Gamma'(\mathcal{A}_{pq}) ],\] where $\Gamma'(\mathcal{A}_{p}) = \overline{K}_{\phi(pq)}$,  $\Gamma'(\mathcal{A}_{q}) = \overline{K}_{\phi(p^2)}$, $\Gamma'(\mathcal{A}_{p^2}) = \overline{K}_{\phi(q)}$ and $\Gamma'(\mathcal{A}_{pq}) = \overline{K}_{\phi(p)}$. It follows that $D_p = \phi(p^2) = p^2-p$ and $D_q = \phi(pq) + \phi(q) = p(q-1)$, $D_{p^2} = \phi(p^2) + \phi(p) = p^2-1$ and $D_{pq} = \phi(q) = q-1$. Therefore, by Theorem \ref{laplacianeigenvaluescozero}, the Laplacian spectrum of $\Gamma'(\mathbb{Z}_{p^2q})$ is 

\begin{align*}
\Phi_{L}(\Gamma'(\mathbb{Z}_{p^2q})) &= (D_{p} + (\Phi_{L}(\Gamma'(\mathcal{A}_{p})) \setminus \{0\})) \bigcup (D_{q} + (\Phi_{L}(\Gamma'(\mathcal{A}_{q})) \setminus \{0\})) \bigcup (D_{p^2} + (\Phi_{L}(\Gamma'(\mathcal{A}_{p^2})) \setminus \{0\})) \\ 
& \bigcup (D_{pq} + (\Phi_{L}(\Gamma'(\mathcal{A}_{pq})) \setminus \{0\})) \bigcup \Phi(L(\Upsilon'_{p^2q}))\\
 &= \displaystyle \begin{pmatrix}
p^2-p & pq-p &  p^2-1 & q-1 \\
(p-1)(q-1)-1& p^2-p-1 & q-2 & p-2 \\
\end{pmatrix} \bigcup \Phi(L(\Upsilon'_{p^2q})).
\end{align*}
Thus, the remaining Laplacian eigenvalues can be obtained by the characteristic polynomial (given in the statement) of the matrix \[L(\Upsilon'_{p^2q})  = \displaystyle \begin{bmatrix}
	p^2-p &  -p^2+p & 0 & 0 \\
	-(p-1)(q-1) & p(q-1) & -(q-1) & 0\\
	0 & -p^2+p & p^2-1 & -p+1 \\
	0 & 0 & -q+1 & q-1 
\end{bmatrix}.\]
\end{proof}

\begin{lemma}
For distinct primes $p$ and $q$, if $n = p^{n_1}q$ then the Laplacian eigenvalues of $\Gamma'(\mathbb{Z}_n)$ consists of the set $\{ \left(\phi(p^{n_1})\right)^{[\phi(p^{n_1-1}q)-1]}, \left(\phi(p^{n_1}) + \phi(p^{n_1-1})\right)^{[\phi(p^{n_1-2}q)-1]}, \left(\sum \limits_{i=0}^{2} \phi(p^{n_1-i})\right)^{[\phi(p^{n_1-3}q) -1]}, \cdots , \left(\sum \limits_{i=0}^{n_1-1} \phi(p^{n_1-i})\right)^{[\phi(p^{n_1-n_1}q) -1]},\\ \left(\sum \limits_{i=1}^{n_1} \phi(p^{n_1-i}q)\right)^{[\phi(p^{n_1}) -1]}, \left(\sum \limits_{i=2}^{n_1} \phi(p^{n_1-i}q)\right)^{[\phi(p^{n_1-1}) -1]}, \cdots , \left(\phi(q)\right)^{[\phi(p)-1]} \}$. \\
and the remaining eigenvalues are the eigenvalues of the matrix given in equation \rm{(\ref{spectrum matrix2}}).
\end{lemma}

\begin{proof}
Note that $\{p, p^2, \cdots, p^{n_1}, q, pq, p^2q, \cdots, p^{n_1-1}q\}$ is the vertex set of the graph $\Upsilon'_{p^{n_1}q}$. By Lemma \ref{inducedsubgraphequaltogamma},
\[ \Gamma'(\mathbb{Z}_{p^{n_1}q}) = \Upsilon'_{p^{n_1}q}[\Gamma'(\mathcal{A}_{p}), \Gamma'(\mathcal{A}_{p^2}), \cdots, \Gamma'(\mathcal{A}_{p^{n_1}}), \Gamma'(\mathcal{A}_{q}), \Gamma'(\mathcal{A}_{pq}), \Gamma'(\mathcal{A}_{p^2q}), \cdots, \Gamma'(\mathcal{A}_{p^{n_1-1}q})],
\]
where, $\Gamma'(\mathcal{A}_{p}) = \overline{K}_{\phi(p^{n_1-1}q)}$, $\Gamma'(\mathcal{A}_{p^2}) = \overline{K}_{\phi(p^{n_1-2}q)}$, $\cdots$, $\Gamma'(\mathcal{A}_{p^{n_1}}) = \overline{K}_{\phi(q)}$, $\Gamma'(\mathcal{A}_{q}) = \overline{K}_{\phi(p^{n_1})}$, $\Gamma'(\mathcal{A}_{pq}) = \overline{K}_{\phi(p^{n_1-1})}$, $\cdots$, $\Gamma'(\mathcal{A}_{p^{n_1-1}q}) = \overline{K}_{\phi(p)}$. It follows that
$D_p = \phi(p^{n_1})$, $D_{p^2} = \phi(p^{n_1}) + \phi(p^{n_1-1})$, $\cdots$, $D_{p^{n_1}} = \sum \limits_{i=0}^{n_1-1}\phi(p^{n_1-i})$,  $D_q = \sum \limits_{i=1}^{n_1}\phi(p^{n_1-i}q)$, $D_{pq} = \sum \limits_{i=2}^{n_1}\phi(p^{n_1-i}q)$, $\cdots$, $D_{p^{n_1-1}q} = \phi(q)$.
Consequently, by Theorem \ref{laplacianeigenvaluescozero}, the Laplacian spectrum of $\Gamma'(\mathbb{Z}_{p^{n_1}q})$ is
\begin{align*}
\Phi_{L}(\Gamma'(\mathbb{Z}_{p^{n_1}q})) &= (D_{p} + (\Phi_{L}(\Gamma'(\mathcal{A}_{p})) \setminus \{0\})) \bigcup (D_{p^2} + (\Phi_{L}(\Gamma'(\mathcal{A}_{p^2})) \setminus \{0\})) \bigcup \cdots \bigcup (D_{p^{n_1}} + (\Phi_{L}(\Gamma'(\mathcal{A}_{p^{n_1}})) \setminus \{0\})) \\
& \ \bigcup (D_{q} + (\Phi_{L}(\Gamma'(\mathcal{A}_{q})) \setminus \{0\})) \bigcup (D_{pq} + (\Phi_{L}(\Gamma'(\mathcal{A}_{pq})) \setminus \{0\})) \bigcup \cdots \bigcup (D_{p^{n_1-1}q} + (\Phi_{L}(\Gamma'(\mathcal{A}_{p^{n_1-1}q})) \setminus \{0\})) \\
& \ \bigcup \Phi(L(\Upsilon'_{p^{n_1}q})). \\
&= \displaystyle \begin{pmatrix}
\phi(p^{n_1}) &  \cdots & \sum \limits_{i=0}^{n_1-1} \phi(p^{n_1-i}) & \sum \limits_{i=1}^{n_1} \phi(p^{n_1-i}q) & \sum \limits_{i=2}^{n_1} \phi(p^{n_1-i}q) & \cdots & \phi(q)\\
\phi(p^{n_1-1}q)-1 & \cdots & \phi(p^{n_1-n_1}q) -1 & \phi(p^{n_1}) -1 & \phi(p^{n_1}) -1 & \cdots & \phi(p)-1
\end{pmatrix} \bigcup \Phi(L(\Upsilon'_{p^{n_1}q})).
\end{align*}

Thus, the remaining $2n_1$ Laplacian eigenvalues are the eigenvalues of the matrix $L(\Upsilon'_{p^{n_1}q}) =  $
\[ \displaystyle \begin{bmatrix}
\phi(p^{n_1}) &  0 & 0 & \cdots & 0 & \cdots & \cdots & 0 & - \phi(p^{n_1})\\
0 &  \sum \limits_{i=0}^{1}\phi(p^{n_1-i}) & 0 & \cdots & \cdots & 0 & \cdots & -\phi(p^{n_1-1}) & - \phi(p^{n_1})\\
\vdots & \vdots &  \vdots & \vdots & \vdots & \vdots & \vdots & \vdots & \vdots \\
0 &  0 & 0 & \cdots & \sum \limits_{i=0}^{n_1-1}\phi(p^{n_1-i}) & - \phi(p) & - \phi(p^2) & \cdots & - \phi(p^{n_1})\\
0 & 0 & 0 & \cdots & -\phi(q) & \phi(q) & 0 &\cdots & 0\\
0 & \cdots& \cdots& -\phi(pq) & -\phi(q) & 0 & \phi(pq)+\phi(q) &  \cdots & 0\\
\vdots & \vdots &  \vdots & \vdots & \vdots & \vdots & \vdots & \vdots & \vdots \\
-\phi(p^{n_1-1}q) & -\phi(p^{n_1-2}q) & \cdots &\cdots &-\phi(q) & 0 & 0 & \cdots & \sum \limits_{i=1}^{n_1}\phi(p^{n_1-i}q)
\end{bmatrix}\]
where matrix $L(\Upsilon'_{p^{n_1}q})$ is obtained by indexing the rows and columns as $p, p^2, \cdots, p^{n_1},  p^{n_1-1}q, \cdots, pq, q$.
\end{proof}

\begin{theorem}
If $n = p^{n_1}q^{n_2}$, where $p$ and $q$ are distinct primes. Then the set of Laplacian eigenvalues of $\Gamma'(\mathbb{Z}_n)$ consists of 
\begin{center}
 $\left(\sum \limits_{i=1}^{n_2}\phi(p^{n_1}q^{n_2-i})\right)^{[\phi(p^{n_1-1}q^{n_2})-1]}, ~~~~~\left(\sum \limits_{i=1}^{n_2}\phi(p^{n_1}q^{n_2-i}) + \sum \limits_{i=1}^{n_2}\phi(p^{n_1-1}q^{n_2-i})\right)^{[\phi(p^{n_1-2}q^{n_2})-1]},$ \\
 
 $\left(\sum \limits_{i=1}^{n_2} \phi(p^{n_1}q^{n_2-i})+ \sum \limits_{i=1}^{n_2}\phi(p^{n_1-1}q^{n_2-i}) + \sum \limits_{i=1}^{n_2}\phi(p^{n_1-2}q^{n_2-i})\right)^{[\phi(p^{n_1-3}q^{n_2}) -1]},$ \\
 
 $\vdots$\\
 
 $\left(\sum \limits_{i=1}^{n_2} \phi(p^{n_1}q^{n_2-i})+ \sum \limits_{i=1}^{n_2}\phi(p^{n_1-1}q^{n_2-i}) + \cdots + \sum \limits_{i=1}^{n_2}\phi(pq^{n_2-i})\right)^{[\phi(q^{n_2}) -1]}, $\\
 
 $\left(\sum \limits_{i=1}^{n_1}\phi(p^{n_1-i}q^{n_2})\right)^{[\phi(p^{n_1}q^{n_2-1})-1]}, \left(\sum \limits_{i=1}^{n_1}\phi(p^{n_1-i}q^{n_2}) + \sum \limits_{i=1}^{n_1}\phi(p^{n_1-i}q^{n_2-1})\right)^{[\phi(p^{n_1}q^{n_2-2})-1]},$\\ 
 
 $\left(\sum \limits_{i=1}^{n_1} \phi(p^{n_1-i}q^{n_2})+ \sum \limits_{i=1}^{n_1}\phi(p^{n_1-i}q^{n_2-1}) + \sum \limits_{i=1}^{n_1}\phi(p^{n_1-i}q^{n_2-2})\right)^{[\phi(p^{n_1}q^{n_2-3}) -1]}, $\\
 
 $\vdots$\\
 
 $\left(\sum \limits_{i=1}^{n_1} \phi(p^{n_1-i}q^{n_2})+ \sum \limits_{i=1}^{n_1}\phi(p^{n_1-i}q^{n_2-1}) + \cdots + \sum \limits_{i=1}^{n_1}\phi(p^{n_1-i}q)\right)^{[\phi(p^{n_1}) -1]},$\\
 
 $\left(\sum \limits_{i=2}^{n_1}\phi(p^{n_1-i}q^{n_2})+ \sum \limits_{i=2}^{n_2}\phi(p^{n_1}q^{n_2-i})\right)^{[\phi(p^{n_1-1}q^{n_2-1})-1]},$\\
 
 $\left(\sum \limits_{i=3}^{n_1}\phi(p^{n_1-i}q^{n_2})+ \sum \limits_{i=2}^{n_2}\phi(p^{n_1}q^{n_2-i}) + \sum \limits_{i=1}^{n_2}\phi(p^{n_1-1}q^{n_2-i})\right)^{[\phi(p^{n_1-2}q^{n_2-1})-1]},$\\
 
 $\vdots$\\
 
  $\left(\sum \limits_{i=2}^{n_2}\phi(p^{n_1}q^{n_2-i}) + \sum \limits_{i=1}^{n_1}\phi(p^{n_1-i}q^{n_2-2}) + \sum \limits_{i=1}^{n_1}\phi(p^{n_1-i}q^{n_2-3})+ \cdots + \sum \limits_{i=1}^{n_1-1}\phi(p^{n_1-i})\right) ^{[\phi(q^{n_2-1})-1]},$\\
  
  $\left(\sum \limits_{i=2}^{n_1}\phi(p^{n_1-i}q^{n_2})+ \sum \limits_{i=3}^{n_2}\phi(p^{n_1}q^{n_2-i}) + \sum \limits_{i=2}^{n_1}\phi(p^{n_1-i}q^{n_2-1})\right)^{[\phi(p^{n_1-1}q^{n_2-2})-1]},$\\
  
  $\vdots$\\
  
  
 
 $\left(\sum \limits_{i=2}^{n_1}\phi(p^{n_1-i}q^{n_2})+ \sum \limits_{i=1}^{n_2}\phi(p^{n_1-2}q^{n_2-i}) + \sum \limits_{i=1}^{n_2}\phi(p^{n_1-3}q^{n_2-i}) + \cdots + \sum \limits_{i=1}^{n_2-1}\phi(q^{n_2-i})\right)^{[\phi(p^{n_1-1})-1]},$\\
 
 $\vdots$\\
 
 $\left(\sum \limits_{i=1}^{n_2-1}\phi(q^{n_2-i}) +\phi(q^{n_2})\right)^{[\phi({p})-1]}$      
\end{center}
and the remaining $(n_1+1)(n_2+1)-2$ eigenvalues are given by the zeros of the characteristic polynomial of the matrix given in equation \rm{(\ref{spectrum matrix2}}). 
\end{theorem}
\begin{proof}
The set of  proper divisors of $n = p^{n_1}q^{n_2}$ is
\begin{center}
$\{p, p^2, \cdots, p^{n_1}, q, q^2, \cdots, q^{n_2}, pq, p^2q, \cdots,$ $p^{n_1}q, pq^2, p^2q^2, \cdots, p^{n_1}q^2,$ $\cdots, pq^{n_2}, p^2q^{n_2}, \cdots, p^{n_1-1}q^{n_2}\}$.
\end{center}
By the definition of $\Upsilon'_n$, note that
\begin{itemize}
\item $p^{i} \sim q^{j}$ for all $i, j$.

\item $p^{i} \sim p^{i_1}q^{j_1}$ for $i > i_1$ and $j_1 > 0$.

\item $q^{j} \sim p^{i}q^{j_1}$ for  $j > j_1$ and $i > 0$.

\item If either  $i_1 > i_2$, $j_1 < j_2$ or $j_1 > j_2$, $i_1 < i_2$, then $p^{i_1}q^{j_1} \sim p^{i_2}q^{j_2}$ .
\end{itemize}
In view of Lemma \ref{inducedsubgraphequaltogamma},
\[
\Gamma'(\mathbb{Z}_{p^{n_1}q^{n_2}}) = \Upsilon'_{p^{n_1}q^{n_2}}[\Gamma'(\mathcal{A}_{p}), \Gamma'(\mathcal{A}_{p^2}), \cdots, \Gamma'(\mathcal{A}_{p^{n_1}}), \Gamma'(\mathcal{A}_{q}), \Gamma'(\mathcal{A}_{q^2}), \cdots, \Gamma'(\mathcal{A}_{q^{n_2}}),\Gamma'(\mathcal{A}_{pq}), \Gamma'(\mathcal{A}_{p^2q}), \cdots, \Gamma'(\mathcal{A}_{p^{n_1}q}),\] \[\cdots, \Gamma'(\mathcal{A}_{pq^{n_2}}),\cdots, \Gamma'(\mathcal{A}_{p^{n_1-1}q^{n_2}})].\]

Therefore, by Lemma \ref{valueof partition} and Corollary \ref{partitionofcozerodivisorgraphisomorphic}, we get 
\begin{center}
$\Gamma'(\mathcal{A}_{p^{i}}) = \overline{K}_{\phi(p^{n_1-i}q^{n_2})}$, where $1 \leq i \leq n_1$,\\
\vspace{0.3cm}
$\Gamma'(\mathcal{A}_{q^j}) = \overline{K}_{\phi(p^{n_1}q^{n_2-j})}$ where $1 \leq j \leq n_2$,\\
\vspace{0.3cm}
$\Gamma'(\mathcal{A}_{p^{i}q^{j}}) = \overline{K}_{\phi(p^{n_1-i}q^{n_2-j})}$.
\end{center}

 Consequently, we have
\begin{center}
$D_ p = \sum \limits_{i=1}^{n_2}\phi(p^{n_1}q^{n_2-i})$,
$D_{p^2} = \sum \limits_{i=1}^{n_2}\phi(p^{n_1}q^{n_2-i}) + \sum \limits_{i=1}^{n_2}\phi(p^{n_1-1}q^{n_2-i})$, \\

$\vdots$\\

$D_{p^{n_1}} = \sum \limits_{i=1}^{n_2} \phi(p^{n_1}q^{n_2-i})+ \sum \limits_{i=1}^{n_2}\phi(p^{n_1-1}q^{n_2-i}) + \cdots + \sum \limits_{i=1}^{n_2}\phi(pq^{n_2-i})$, \\

$D_q = \sum \limits_{i=1}^{n_1}\phi(p^{n_1-i}q^{n_2})$,\\

$\vdots$\\

$D_{q^{n_2}} = \sum \limits_{i=1}^{n_1} \phi(p^{n_1-i}q^{n_2})+ \sum \limits_{i=1}^{n_1}\phi(p^{n_1-i}q^{n_2-1}) + \cdots + \sum \limits_{i=1}^{n_1}\phi(p^{n_1-i}q)$, $D_{pq} = \sum \limits_{i=2}^{n_1}\phi(p^{n_1-i}q^{n_2})+ \sum \limits_{i=2}^{n_2}\phi(p^{n_1}q^{n_2-i})$, \\

$\vdots$\\

$D_{p^{n_1}q} = \sum \limits_{i=2}^{n_2}\phi(p^{n_1}q^{n_2-i}) + \sum \limits_{i=1}^{n_1}\phi(p^{n_1-i}q^{n_2-2}) + \sum \limits_{i=1}^{n_1}\phi(p^{n_1-i}q^{n_2-3})+ \cdots + \sum \limits_{i=1}^{n_1-1}\phi(p^{n_1-i})$, $D_{pq^2} = \sum \limits_{i=2}^{n_1}\phi(p^{n_1-i}q^{n_2})+ \sum \limits_{i=3}^{n_2}\phi(p^{n_1}q^{n_2-i}) + \sum \limits_{i=2}^{n_1}\phi(p^{n_1-i}q^{n_2-1})$,\\

$\vdots$\\

$D_{pq^{n_2}} = \sum \limits_{i=2}^{n_1}\phi(p^{n_1-i}q^{n_2})+ \sum \limits_{i=1}^{n_2}\phi(p^{n_1-2}q^{n_2-i}) + \sum \limits_{i=1}^{n_2}\phi(p^{n_1-3}q^{n_2-i}) + \cdots + \sum \limits_{i=1}^{n_2-1}\phi(q^{n_2-i})$,\\

$\vdots$\\

$D_{p^{n_1-1}q^{n_2}} = \sum \limits_{i=1}^{n_2-1}\phi(q^{n_2-i}) +\phi(q^{n_2})$. 
\end{center}
Therefore, by Theorem \ref{laplacianeigenvaluescozero}, the  Laplacian spectrum of  $\Gamma'(\mathbb{Z}_{p^{n_1}q^{n_2}})$ is
\begin{align*}
\Phi_{L}(\Gamma'(\mathbb{Z}_{p^{n_1}q^{n_2}})) & = (D_{p} + (\Phi_{L}(\Gamma'(\mathcal{A}_{p})) \setminus \{0\})) \bigcup (D_{p^2} + (\Phi_{L}(\Gamma'(\mathcal{A}_{p^2})) \setminus \{0\})) \bigcup \cdots \bigcup (D_{p^{n_1}} + (\Phi_{L}(\Gamma'(\mathcal{A}_{p^{n_1}})) \setminus \{0\}))\\
 & \bigcup (D_{q} + (\Phi_{L}(\Gamma'(\mathcal{A}_{q})) \setminus \{0\})) \bigcup (D_{q^2} + (\Phi_{L}(\Gamma'(\mathcal{A}_{q^2})) \setminus \{0\})) \bigcup \cdots \bigcup (D_{q^{n_2}} + (\Phi_{L}(\Gamma'(\mathcal{A}_{q^{n_2}})) \setminus \{0\})) \\
 & \bigcup (D_{pq} + (\Phi_{L}(\Gamma'(\mathcal{A}_{pq})) \setminus \{0\})) \bigcup \cdots \bigcup (D_{p^{n_1}q} + (\Phi_{L}(\Gamma'(\mathcal{A}_{p^{n_1}q})) \setminus \{0\})) \bigcup \cdots\\ & \bigcup  (D_{pq^{n_2}} + (\Phi_{L}(\Gamma'(\mathcal{A}_{pq^{n_2}})) \setminus \{0\})) 
  \bigcup \cdots \bigcup (D_{p^{n_1-1}q^{n_2}} + (\Phi_{L}(\Gamma'(\mathcal{A}_{p^{n_1-1}q^{n_2}})) \setminus \{0\})) \bigcup \Phi(L(\Upsilon'_{p^{n_1}q^{n_2}})).
\end{align*}
The remaining $(n_1+1)(n_2+1)-2$ eigenvalues are the zeros of the characteristic polynomial of the matrix $L(\Upsilon'_{p^{n_1}q^{n_2}})$ given in equation (\ref{spectrum matrix2}).
\end{proof}

\section{The Laplacian spectral radius and the Algebraic connectivity of $\Gamma'(\mathbb{Z}_n)$}
In this section, we study the algebraic connectivity and the Laplacian spectral radius of $\Gamma'(\mathbb{Z}_n)$. We obtain all those values of $n$ for which the Laplacian spectral radius of $\Gamma'(\mathbb{Z}_n)$ is equal to order of $\Gamma'(\mathbb{Z}_n)$. Moreover, the values of $n$ for which the algebraic connectivity and the vertex connectivity coincide are also described. The following theorem follows from the relation $\lambda(\Gamma) = |V(\Gamma)| - \mu(\overline{\Gamma})$ and the fact $\overline{\Gamma}$ is disconnected if and only if $\Gamma$ is the join of two graphs.

\begin{theorem}[\cite{fiedler1973algebraic}]\label{Spectralradius} 
If $\Gamma$ is a graph on m vertices, then $\lambda(\Gamma) \le m $. Further, equality holds if and only if $\overline{\Gamma}$ is disconnected if and only if $\Gamma$ is the join of two graphs.
\end{theorem}

In view of Theorem \ref{Spectralradius}, first we characterize the values of $n$ for which the complement of $\Gamma'(\mathbb{Z}_n)$ is disconnected.

\begin{proposition}\label{connectednessofcomplement}
$\overline{\Gamma'(\mathbb{Z}_n)}$ is disconnected if and only if $n$ is a product of two distinct primes.
\end{proposition}

\begin{proof}
Let $p$ and $q$ be two distinct primes. If $n= pq$, then by Remark \ref{partition} we get $V(\Gamma'(\mathbb{Z}_n)) = \mathcal{A}_p \cup \mathcal{A}_q$ such that $\mathcal{A}_p \cap \mathcal{A}_q = \emptyset $. In fact, $\Gamma'(\mathbb{Z}_n) = K_{\phi (q), \phi(p)}$ is a complete bipartite graph. Consequently, $\overline{\Gamma'(\mathbb{Z}_n)}$ is a disconnected graph.

Conversely, suppose $\overline{\Gamma'(\mathbb{Z}_n)}$ is disconnected. Clearly, for $n = p$ there is nothing to prove. If $n = p^\alpha$ for some $1 < \alpha \in \mathbb{N}$, then $\Gamma'(\mathbb{Z}_{p^\alpha})$ is a null graph. Consequently, $\overline{\Gamma'(\mathbb{Z}_{p^\alpha})}$ is a complete graph which is not possible.  If possible, let $n \ne pq$. Let $d_1$ and $d_2$ be the proper divisors of $n$ and let $x \in  \mathcal{A}_{d_1} $, $y \in  \mathcal{A}_{d_2} $. If $d_1 = d_2$ then clearly $x \sim y$ in  $\overline{\Gamma'(\mathbb{Z}_n)}$. If $ d_1 \ne d_2$ such that either $d_1 \mid d_2$ or $d_2 \mid d_1$ then $x \sim y$ in $\overline{\Gamma'(\mathbb{Z}_n)}$ (cf. Lemma \ref{adjacenyofvertex} ). If $d_1 \ne d_2$ and neither $d_1 \mid d_2$ nor $d_2 \mid d_1$, then there exist two primes $p_1$ and $p_2$ such that $p_1 \mid d_1$ and $p_2 \mid d_2$. Consequently, $x \sim z_1 \sim z_2 \sim z_3 \sim y$ in $\overline{\Gamma'(\mathbb{Z}_n)}$ for some $z_1 \in \mathcal{A}_{p_1}$, $z_2 \in \mathcal{A}_{p_1p_2}$ and $z_3 \in \mathcal{A}_{p_2}$. Thus, $\overline{\Gamma'(\mathbb{Z}_n)}$ is connected; a contradiction. Hence, $n$ must be a product of two distinct primes.
\end{proof}

Since $| V(\Gamma'(\mathbb{Z}_n))| = n - \phi(n) - 1$,  by using the Proposition \ref{connectednessofcomplement} in Theorem \ref{Spectralradius}, we have the following proposition.

\begin{proposition}
$\lambda(\Gamma'(\mathbb{Z}_n)) = | V(\Gamma'(\mathbb{Z}_n))|$ if and only if $n$ is a product of two distinct primes. Moreover, if $n = pq$ then $\lambda(\Gamma'(\mathbb{Z}_n)) = p + q - 2.$
\end{proposition}

Now we classify all those values of $n$ for which the algebraic connectivity and the vertex connectivity of $\Gamma'(\mathbb{Z}_n)$ are equal. The following theorem is useful in this study.

\begin{theorem}\cite{kirkland2002graphs}\label{equality of vertex and algebraic connectivity}
Let $\Gamma$ be a non-complete connected graph on $m$ vertices. Then $\kappa(\Gamma) = \mu(\Gamma)$ if and only if $\Gamma$ can be written as $\Gamma_1 \vee \Gamma_2$, where $\Gamma_1$ is a disconnected graph on $m- \kappa(\Gamma)$ vertices and $\Gamma_2$ is a graph on $\kappa(\Gamma)$ vertices with $\mu(\Gamma_2) \ge 2\kappa(\Gamma)-m$.
\end{theorem}

\begin{lemma} \label{vertexconnectivityofn = pq}
For distinct primes $p$ and $q$, if $n=pq$ where $p<q$ then $\kappa(\Gamma'(\mathbb{Z}_n)) = \delta(\Gamma'(\mathbb{Z}_n)) = p-1 $.
\end{lemma}

\begin{proof}
For $n = pq$, $\Gamma'(\mathbb{Z}_n)$ is a complete bipartite graph with partition sets $\mathcal{A}_p$ and $\mathcal{A}_q$. Hence, $\kappa(\Gamma'(\mathbb{Z}_n)) = \delta(\Gamma'(\mathbb{Z}_n)) = \text{min}\{ |\mathcal{A}_p|, |\mathcal{A}_q| \} = p-1$
\end{proof}

\begin{theorem}
For the graph $\Gamma'(\mathbb{Z}_n)$, we have $\mu(\Gamma'(\mathbb{Z}_n)) \le \kappa(\Gamma'(\mathbb{Z}_n))$. The equality holds if and only if $n$ is a product of two distinct primes.
\end{theorem}

\begin{proof}
By \cite{kirkland2002graphs}, for any graph $\Gamma$ which is not complete, we have  $\mu(\Gamma) \le \kappa(\Gamma)$. If $n=4$ then there is nothing to prove because $\Gamma'(\mathbb{Z}_4)$ is the graph of one vertex only. If $n \ne 4$ then $\Gamma'(\mathbb{Z}_n)$ is not a complete graph. Consequently, $\mu(\Gamma'(\mathbb{Z}_n)) \le \kappa(\Gamma'(\mathbb{Z}_n))$.

If $n$ is not a product of two distinct primes then by Proposition \ref{connectednessofcomplement} and by Theorem \ref{Spectralradius}, $\Gamma'(\mathbb{Z}_n)$ can not be written as the join of two graphs. Thus, by Theorem \ref{equality of vertex and algebraic connectivity}, we obtain $\mu(\Gamma'(\mathbb{Z}_n)) < \kappa(\Gamma'(\mathbb{Z}_n))$. If $n=pq$, where $p$ and $q$ are distinct primes such that $p<q$, then by Theorem \ref{Spectralradius}, Proposition \ref{connectednessofcomplement}, Theorem \ref{equality of vertex and algebraic connectivity} and Lemma \ref{vertexconnectivityofn = pq}, we obtain $\mu(\Gamma'(\mathbb{Z}_n)) = \kappa(\Gamma'(\mathbb{Z}_n)) = p-1$.  
\end{proof}
\section{The Wiener index of $\Gamma'(\mathbb{Z}_n)$}
In this section, we obtain the Wiener index of the cozero-divisor graph of the ring $\mathbb{Z}_{n}$ for arbitrary $n \in \mathbb{N}$. Consequently, we obtain the diameter of $\Gamma'(\mathbb{Z}_n)$ (see Proposition \ref{diameterpro}). For a prime $p$ and $1 \le \alpha \in \mathbb{N}$,  the graph $\Gamma'(\mathbb{Z}_p)$ is empty whereas $\Gamma'(\mathbb{Z}_{p^{\alpha}})$ is a null graph. Therefore $W(\Gamma'(\mathbb{Z}_p)) = W(\Gamma'(\mathbb{Z}_{p^\alpha})) = 0$.

\begin{theorem}\label{n=p_1p_2...p_k}
For $ 1 \le i \le \tau(n)-2$, let $d_i$'s be the proper divisors of $n$. If $n = p_1p_2 \cdots p_k$, where $p_i$'s are distinct primes and $2 \le k \in \mathbb{N}$, then \[W(\Gamma'(\mathbb{Z}_n)) = \sum_{i=1}^{2^k-2} \phi(\dfrac{n}{d_i})\left(\phi(\dfrac{n}{d_i})-1\right) \ + \  \dfrac{1}{2}\sum_{\substack{d_i \nmid d_j \\ d_j \nmid d_i}} \phi(\dfrac{n}{d_i})\phi(\dfrac{n}{d_j}) \ + \ 2\sum_{\substack{d_i \mid d_j \\ i\neq j}} \phi(\dfrac{n}{d_i})\phi(\dfrac{n}{d_j}).\]
\end{theorem}

\begin{proof}
To determine the Wiener index of $\Gamma'(\mathbb{Z}_n)$, we first obtain the distances between the vertices of each $\mathcal{A}_{d_i}$ and two distinct $\mathcal{A}_{d_i}$'s, respectively. For a proper divisor $d_i$ of $n$, let $x, y \in \mathcal{A}_{d_i}$. Since $\Gamma'(\mathbb{Z}_n)$ is connected, by Corollary \ref{partitionofcozerodivisorgraphisomorphic}, there exists a proper divisor $d_r$ of $n$ such that $x \sim z$ for each $x \in \mathcal{A}_{d_i}$ and $z \in \mathcal{A}_{d_r}$. Consequently, $d(x,y) = 2$ for any two distinct $x, y \in \mathcal{A}_{d_i}$. Now we obtain the distances between the vertices of any two distinct $\mathcal{A}_{d_i}$'s through the following cases.

\textbf{Case-1}: Neither $d_i\mid d_j$ nor $d_j \mid d_i$. By Lemma \ref{adjacenyofvertex}, $d(x,y) =1$ for every $x \in \mathcal{A}_{d_i}$ and $y \in \mathcal{A}_{d_j}$.

\textbf{Case-2}: $d_i\mid d_j$. For $x \in \mathcal{A}_{d_i}$ and $y \in \mathcal{A}_{d_j}$ we have $x \nsim y$. Without loss of generality, assume that $d_j = p_1p_2 \cdots p_md_i$, where $1 \leq m \leq k-2$. Since $d_j$ is a proper divisor of $n$ there exists a prime $p$ such that $p \nmid d_j$. Consequently, $p \nmid d_i$. It follows that for $x \in \mathcal{A}_{d_i}$ and $y \in \mathcal{A}_{d_j}$ there exists a $z \in \mathcal{A}_{p}$ such that $x \sim z \sim y$. Thus, $d(x,y)=2$ for each $x \in \mathcal{A}_{d_i}$ and $y \in \mathcal{A}_{d_j}$.\\

Thus, in view of all the possible distances between the vertices of $\Gamma'(\mathbb{Z}_n)$, we get
\begin{align*}
  W(\Gamma'(\mathbb{Z}_n)) &= \dfrac{1}{2}\sum_{u \in V(\Gamma'(\mathbb{Z}_n))}\sum_{v \in V(\Gamma'(\mathbb{Z}_n))} d(u,v)  \\
&=\dfrac{1}{2} \left[ \sum_{i=1}^{2^k-2} 2 |\mathcal{A}_{d_i}| \left(|\mathcal{A}_{d_i}| -1 \right) \ + \ \sum_{\substack{d_i \nmid d_j \\ d_j \nmid d_i}}|\mathcal{A}_{d_i}||\mathcal{A}_{d_j}| \right] \ + \ 2\sum_{\substack{d_i \mid d_j \\ i\neq j}}|\mathcal{A}_{d_i}||\mathcal{A}_{d_j}| \\
&= \sum_{i=1}^{2^k-2} \phi(\dfrac{n}{d_i})\left(\phi(\dfrac{n}{d_i})-1\right) \ + \  \dfrac{1}{2}\sum_{\substack{d_i \nmid d_j \\ d_j \nmid d_i}} \phi(\dfrac{n}{d_i})\phi(\dfrac{n}{d_j}) \ + \ 2\sum_{\substack{d_i \mid d_j \\ i\neq j}} \phi(\dfrac{n}{d_i})\phi(\dfrac{n}{d_j}).
\end{align*}
\end{proof}

\begin{corollary}
If $n = pq$, where $p, q$ are distinct primes, then $W(\Gamma'(\mathbb{Z}_n)) = (p-1)(q-1) + (p-1)(p-2) + (q-1)(q-2)$.
\end{corollary}

\begin{theorem}\label{general n}
Let $n = p_1^{n_1}p_2^{n_2}\cdots p_r^{n_r} \cdots p_k^{n_k}$ with $k \geq 2$, where $p_i$'s are distinct primes and let $D= \{ d_1, d_2, \ldots ,d_{\tau(n)-2} \}$ be the set of all proper divisors of $n$. For $d_i \mid d_j$, define
\begin{align*}
A &= \{ (d_i, d_j) \in D \times D \  | \ d_i \neq p_r^s \}; \\
B &= \{ (d_i, d_j) \in D \times D \  | \ d_i = p_r^s \ \text{and} \ \dfrac{n}{d_j} \neq p_r^t \}; \\
C &= \{ (d_i, d_j) \in D \times D \  | \ d_i = p_r^s \ \text{and} \  \dfrac{n}{d_j} = p_r^t \}.
\end{align*}
Then \[W(\Gamma'(\mathbb{Z}_n)) = \sum_{i=1}^{\tau (n)-2} \phi(\dfrac{n}{d_i}) \left(\phi(\dfrac{n}{d_i}) -1 \right) \ + \ \dfrac{1}{2}\sum_{\substack{d_i \nmid d_j \\ d_j \nmid d_i}}\phi(\dfrac{n}{d_i})\phi(\dfrac{n}{d_j}) \ + \ 2\sum_{(d_i, d_j) \in A}\phi(\dfrac{n}{d_i})\phi(\dfrac{n}{d_j}) 
 + \ 2\sum_{(d_i, d_j) \in B}\phi(\dfrac{n}{d_i})\phi(\dfrac{n}{d_j})\]  
 \[ \hspace{-10cm} + ~~~         3\sum_{(d_i, d_j) \in C}\phi(\dfrac{n}{d_i})\phi(\dfrac{n}{d_j}).\]
\end{theorem}

\begin{proof}
In view of Remark \ref{partition}, first we obtain all the possible distances between the vertices of $\mathcal{A}_{d_i}$ and $\mathcal{A}_{d_j}$, where $d_i$ and $d_j$ are proper divisors of $n$. If $i=j$ then by the proof of Theorem \ref{n=p_1p_2...p_k}, we get $d(x,y)=2$ for any two distinct $x,y \in \mathcal{A}_{d_i}$. Now suppose that $i \neq j$. If $d_i \nmid d_j$ and $d_j \nmid d_i$, then by Lemma \ref{adjacenyofvertex}, we get $d(x,y)=1$ for every $x \in \mathcal{A}_{d_i}$ and $y \in \mathcal{A}_{d_j}$. If $d_i \mid d_j$ then we obtain the possible distances through the following cases.

\textbf{Case-1}: $(d_i, d_j) \in A$. Since $d_i \mid d_j$, we have $x \nsim y$ for any $x \in \mathcal{A}_{d_i}$ and $y \in \mathcal{A}_{d_j}$. Note that $d_i \neq p_r^s$ implies that $d_i = p_1^{\beta_1}p_2^{\beta_2} \cdots p_m^{\beta_m}$ for some $\beta_i$'s $\in \mathbb{N} \cup \{ 0 \}$ and $m \geq 2$. Consequently, $d_j = p_1^{\alpha_1}p_2^{\alpha_2} \cdots p_m^{\alpha_m}$ for some $\alpha_i$'s $\in \mathbb{N} \cup \{ 0 \}$. Since $d_j$ is a proper divisor of $n$ there exists $l \in \{1, 2, \ldots, k\}$ such that $p_l^{n_l} \nmid d_j$. Also, $p_l^{n_l} \nmid d_i$. Further, $m \geq 2$ follows that $d_i \nmid p_l^{n_l}$ and $d_j \nmid p_l^{n_l}$. Now for any $x \in \mathcal{A}_{d_i}$, $y \in \mathcal{A}_{d_j}$ there exists a $z \in \mathcal{A}_{p_l^{n_l}}$ such that $x \sim z \sim y$. Thus $d(x,y)=2$ for every $x \in \mathcal{A}_{d_i}$ and $y \in \mathcal{A}_{d_j}$. 

\textbf{Case-2}: $d_i = p_r^s$ for some $r \in \{ 1,2, \ldots ,k \}$ and $1 \leq s \leq n_r$. Suppose $x\in \mathcal{A}_{d_i}$ and $y \in \mathcal{A}_{d_j}$. Then we obtain $d(x,y)$ in the following subcases: 

\hspace{.5cm}\textbf{Subcase-2.1}: $(d_i, d_j) \in B$. Suppose $d_j = p_1^{\alpha_1}p_2^{\alpha_2} \cdots p_r^{\alpha_r} \cdots p_k^{\alpha_k}$. Since $\dfrac{n}{d_j} \neq p_r^t$ there exists a prime $p_m \in \{ p_1, p_2, \ldots, p_k \} \setminus \{p_r\}$ and $\alpha_m \leq n_m$ such that $p_m^{n_m} \nmid d_j$. Consequently, $p_m^{n_m} \nmid d_i$. Moreover, $d_i \nmid p_m^{n_m}$ and $ d_j \nmid p_m^{n_m}$. Thus, for every $x \in \mathcal{A}_{d_i}$ and $y \in \mathcal{A}_{d_j}$, we get $x \sim z \sim y$ for some $z \in \mathcal{A}_{p_m^{n_m}}$. Hence, $d(x,y) =2$ for each $x \in \mathcal{A}_{d_i}$ and $y \in \mathcal{A}_{d_j}$. 

\hspace{.5cm}\textbf{Subcase-2.2}: $(d_i, d_j) \in C$. Then $d_j = p_1^{n_1}p_2^{n_2} \cdots p_r^{\alpha_r} \cdots  p_k^{n_k}$, where $n_r - \alpha_r =t \ge 1$. Since $d_i \mid d_j$, for each $x \in \mathcal{A}_{d_i}$ and $y \in \mathcal{A}_{d_j}$, we have $d(x,y) \geq 2$ (cf. Lemma \ref{adjacenyofvertex}). First, we show that $d(x,y) > 2$ for any $x \in \mathcal{A}_{d_i}$ and $y \in \mathcal{A}_{d_j}$. In this connection, it is sufficient to prove that for any proper divisor $d$ of $n$, we have either $d \mid d_j$ or $d_i \mid d$. Suppose that $d \nmid d_j$. Then $d = p_1^{\gamma_1}p_2^{\gamma_2} \cdots p_r^{\gamma_r} \cdots  p_k^{\gamma_k}$ together with $\gamma_r > \alpha_r$. Since $p_r^s = d_i \mid d_j$, we get $s\leq \alpha_r < \gamma_r$. Consequently, $d_i \mid d$. 

Since $n = p_1^{n_1}p_2^{n_2}\cdots p_r^{n_r} \cdots p_k^{n_k}$ with $k\geq2$, there exists a prime $q \neq p_r$ such that $q \mid n$. Clearly, $q \nmid d_i$ and $d_i \nmid q$. Also, $p_r^{n_r} \nmid q$ and $q \nmid p_r^{n_r}$. Since $\alpha_r < n_r$, we obtain $d_j \nmid p_r^{n_r}$ and $p_r^{n_r} \nmid d_j$. Thus, in view of Lemma \ref{adjacenyofvertex}, for any $x \in \mathcal{A}_{d_i}$ and $y \in \mathcal{A}_{d_j}$, there exist $z \in \mathcal{A}_q $ and $w \in \mathcal{A}_{p_r^{n_r}}$ such that $x \sim z \sim w \sim y$. Hence, $d(x,y) = 3$ for every $x \in \mathcal{A}_{d_i}$ and $y \in \mathcal{A}_{d_j}$. 
In view of the cases and arguments discussed in this proof, we have
\[W(\Gamma'(\mathbb{Z}_n)) = \dfrac{1}{2}\sum_{u \in V(\Gamma'(\mathbb{Z}_n))}\sum_{v \in V(\Gamma'(\mathbb{Z}_n)} d(u,v)  \]
\begin{align*}
&=\dfrac{1}{2} \left[ \sum_{i=1}^{\tau (n)-2} 2 |\mathcal{A}_{d_i}| \left(|\mathcal{A}_{d_i}| -1 \right) \ + \ \sum_{\substack{d_i \nmid d_j \\ d_j \nmid d_i}}|\mathcal{A}_{d_i}||\mathcal{A}_{d_j}| \right] \ + \ 2\sum_{(d_i, d_j) \in A}|\mathcal{A}_{d_i}||\mathcal{A}_{d_j}| \ + \ 2\sum_{(d_i, d_j) \in B}|\mathcal{A}_{d_i}||\mathcal{A}_{d_j}| \\
&\ \ \ + 3\sum_{(d_i, d_j) \in C}|\mathcal{A}_{d_i}||\mathcal{A}_{d_j}| \\
&= \sum_{i=1}^{\tau (n)-2} \phi(\dfrac{n}{d_i}) \left(\phi(\dfrac{n}{d_i}) -1 \right) \ + \ \dfrac{1}{2}\sum_{\substack{d_i \nmid d_j \\ d_j \nmid d_i}}\phi(\dfrac{n}{d_i})\phi(\dfrac{n}{d_j}) \ + \ 2\sum_{(d_i, d_j) \in A}\phi(\dfrac{n}{d_i})\phi(\dfrac{n}{d_j}) \\
&\ \ \ + \ 2\sum_{(d_i, d_j) \in B}\phi(\dfrac{n}{d_i})\phi(\dfrac{n}{d_j}) \ + \ 3\sum_{(d_i, d_j) \in C}\phi(\dfrac{n}{d_i})\phi(\dfrac{n}{d_j}).
\end{align*}
\end{proof}

Based on all the possible distances obtained in this section, the following proposition is easy to observe.

\begin{proposition}\label{diameterpro}
The diameter of $\Gamma'(\mathbb{Z}_n)$ is given below:
\[{\rm diam}(\Gamma'(\mathbb{Z}_n)) =
  \begin{cases}
0  & \;\; \;n = 4, \\
2  & \;\;\;n = p_1p_2\cdots p_k,~ k \geq 2\\
3 & \;\;\;{\rm otherwise.}
  \end{cases}\]
\end{proposition}

Now we conclude this paper with an illustration of Theorem \ref{general n} for $n= 72$.

\begin{example}
Consider $n= 2^3 \cdot 3^2 = 72$. Then the number of proper divisor $\tau (n)$ of $n$ is $\prod_{i=1}^{k} (n_i + 1)- 2 = 10$. Therefore, $D= \{ 2, \ 2^2, \ 2^3, \ 3, \ 3^2, \ 2 \cdot 3, \ 2^2 \cdot 3, \ 2^3 \cdot 3, \ 2 \cdot 3^2, \ 2^2 \cdot 3^2 \}$. Let $ d_1=2, \ d_2=2^2, \ d_3=2^3, \ d_4=3, \ d_5=3^2, \ d_6=2 \cdot 3, \ d_7=2^2 \cdot 3, \ d_8=2^3 \cdot 3, \ d_9=2 \cdot 3^2, \ d_{10}=2^2 \cdot 3^2 $. By Lemma \ref{valueof partition},  we obtain $|\mathcal{A}_{d_1}| = 12, \ |\mathcal{A}_{d_2}| = 6, \ |\mathcal{A}_{d_3}| = 6, \ |\mathcal{A}_{d_4}| = 8, \ |\mathcal{A}_{d_5}| = 4, \ |\mathcal{A}_{d_6}| = 4, \ |\mathcal{A}_{d_7}| = 2, \ |\mathcal{A}_{d_8}| = 2, \ |\mathcal{A}_{d_9}| = 2, \ |\mathcal{A}_{d_{10}}| = 1$. Now
\[\dfrac{1}{2}\sum_{i=1}^{10} 2 |\mathcal{A}_{d_i}| \left(|\mathcal{A}_{d_i}| -1 \right) = [132+30+30+56+12+12+2+2+2+0] = 278\]
and
\[  \dfrac{1}{2}\sum_{\substack{d_i \nmid d_j \\ d_j \nmid d_i}}|\mathcal{A}_{d_i}||\mathcal{A}_{d_j}| = [ 96+48+48+24+24+12+48+24+24+12+12+6+16+8+8+4+4+2]  = 420\]
The sets $A, \ B$ and  $C $ defined in Theorem \ref{general n} are
 \begin{align*}
     A &= \{ (d_6, d_7),\ (d_6, d_8),\ (d_6, d_9),\ (d_6, d_{10}),\ (d_7, d_8),\ (d_7, d_{10}),\ (d_9, d_{10}) \}; \\
     B &= \{ (d_1, d_2),\ (d_1, d_3),\ (d_1, d_6), \ (d_1, d_7),\ (d_1, d_8),\ (d_2, d_3),\ (d_2, d_7),\ (d_2, d_8),\ (d_3, d_8),\ (d_4, d_5),\ (d_4, d_6),\ (d_4, d_7),\\
     &\ \ \ \ \ (d_4, d_9),\ (d_4, d_{10}),\ (d_5, d_9),\ (d_5, d_{10})\}; \\
     C &= \{ (d_1, d_9),\ (d_1, d_{10}),\ (d_2, d_{10}),\ (d_4, d_8)\}.
 \end{align*}
 Consequently,
 \begin{align*}
2\sum_{(d_i, d_j) \in A}|\mathcal{A}_{d_i}||\mathcal{A}_{d_j}| &= 2[ 8+8+8+4+4+2+2]  = 72 \\
2\sum_{(d_i, d_j) \in B}|\mathcal{A}_{d_i}||\mathcal{A}_{d_j}| &= 2[ 72+72+48+24+24+36+12+12+12+32+32+16+16+8+8+4]  = 856\\
3\sum_{(d_i, d_j) \in C}|\mathcal{A}_{d_i}||\mathcal{A}_{d_j}| &= 3[24+12+6+16]  = 174
\end{align*} Hence, the Wiener index of $\Gamma'(\mathbb{Z}_{72})$ is given by \\
 \begin{align*}
     W(\Gamma(\mathbb{Z}_{72})) &= \dfrac{1}{2} \left[ \sum_{i=1}^{\tau (n)-2} 2 |\mathcal{A}_{d_i}| \left(|\mathcal{A}_{d_j}| -1 \right) \ + \ \sum_{\substack{d_i \nmid d_j \\ d_j \nmid d_i}}|\mathcal{A}_{d_i}||\mathcal{A}_{d_j}| \right] \ + \ 2\sum_{(d_i, d_j) \in A}|\mathcal{A}_{d_i}||\mathcal{A}_{d_j}| \ + \ 2\sum_{(d_i, d_j) \in B}|\mathcal{A}_{d_i}||\mathcal{A}_{d_j}| \\
&\ \ \ + 3\sum_{(d_i, d_j) \in C}|\mathcal{A}_{d_i}||\mathcal{A}_{d_j}| \\
&= 278+420+72+856+174 = 1800.
 \end{align*}
\end{example}

\section*{Declarations} We are thankful to the referee for valuable suggestions which helped in improving the presentation of the paper. 
\vspace{.3cm}

\textbf{Funding}: The second author gratefully acknowledge for providing financial support to CSIR  (09/719(0093)/2019-EMR-I)) government of India. 

\vspace{.3cm}
\textbf{Conflicts of interest/Competing interests}: There is no conflict of interest regarding the publishing of this paper. 

\vspace{.3cm}
\textbf{Availability of data and material (data transparency)}: Not applicable.


\begin{thebibliography}{10}

\bibitem{afkhami2011cozero}
M.~Afkhami and K.~Khashyarmanesh.
\newblock The cozero-divisor graph of a commutative ring.
\newblock {\em Southeast Asian Bulletin of Mathematics}, 35(5), 2011.

\bibitem{afkhami2012cozero}
M.~Afkhami and K.~Khashyarmanesh.
\newblock On the cozero-divisor graphs of commutative rings and their
  complements.
\newblock {\em Bulletin of the Malaysian Mathematical Sciences Society}, 35(4),
  2012.

\bibitem{afkhami2012planar}
M.~Afkhami and K.~Khashyarmanesh.
\newblock Planar, outerplanar, and ring graph of the cozero-divisor graph of a
  finite commutative ring.
\newblock {\em Journal of Algebra and its Applications}, 11(06):1250103, 2012.

\bibitem{afkhami2013cozero}
M.~Afkhami and K.~Khashyarmanesh.
\newblock On the cozero-divisor graphs and comaximal graphs of commutative
  rings.
\newblock {\em Journal of Algebra and its Applications}, 12(03):1250173, 2013.

\bibitem{akbari2014some}
S.~Akbari, F.~Alizadeh, and S.~Khojasteh.
\newblock Some results on cozero-divisor graph of a commutative ring.
\newblock {\em Journal of Algebra and its Applications}, 13(03):1350113, 2014.

\bibitem{akbari2014commutative}
S.~Akbari and S.~Khojasteh.
\newblock Commutative rings whose cozero-divisor graphs are unicyclic or of
  bounded degree.
\newblock {\em Communications in Algebra}, 42(4):1594--1605, 2014.

\bibitem{asir2021wiener}
T.~Asir and V.~Rabikka.
\newblock The wiener index of the zero-divisor graph of $\mathbb{Z}_n$.
\newblock {\em To appear, Discrete Applied Mathematics}, 2021.

\bibitem{bakhtyiari2020coloring}
M.~Bakhtyiari, R.~Nikandish, and M.~Nikmehr.
\newblock Coloring of cozero-divisor graphs of commutative von neumann regular
  rings.
\newblock {\em Proceedings-Mathematical Sciences}, 130(1):1--7, 2020.

\bibitem{beck1988coloring}
I.~Beck.
\newblock Coloring of commutative rings.
\newblock {\em Journal of algebra}, 116(1):208--226, 1988.

\bibitem{cardoso2013spectra}
D.~M. Cardoso, M.~A.~A. de~Freitas, E.~A. Martins, and M.~Robbiano.
\newblock Spectra of graphs obtained by a generalization of the join graph
  operation.
\newblock {\em Discrete Mathematics}, 313(5):733--741, 2013.

\bibitem{chattopadhyay2020laplacian}
S.~Chattopadhyay, K.~L. Patra, and B.~K. Sahoo.
\newblock Laplacian eigenvalues of the zero divisor graph of the ring
  $\mathbb{Z}_n$.
\newblock {\em Linear Algebra and its Applications}, 584:267--286, 2020.

\bibitem{dobrynin2001wiener}
A.~A. Dobrynin, R.~Entringer, and I.~Gutman.
\newblock Wiener index of trees: theory and applications.
\newblock {\em Acta Applicandae Mathematica}, 66(3):211--249, 2001.

\bibitem{fiedler1973algebraic}
M.~Fiedler.
\newblock Algebraic connectivity of graphs.
\newblock {\em Czechoslovak mathematical journal}, 23(2):298--305, 1973.

\bibitem{janezic2015graph}
D.~Janezic, A.~Milicevic, S.~Nikolic, and N.~Trinajstic.
\newblock {\em Graph-theoretical matrices in chemistry}.
\newblock CRC Press, 2015.

\bibitem{kirkland2002graphs}
S.~J. Kirkland, J.~J. Molitierno, M.~Neumann, and B.~L. Shader.
\newblock On graphs with equal algebraic and vertex connectivity.
\newblock {\em Linear algebra and its applications}, 341(1-3):45--56, 2002.

\bibitem{magi2020spectrum}
P.~Magi, S.~M. Jose, and A.~Kishore.
\newblock Spectrum of the zero-divisor graph on the ring of integers modulo $n$.
\newblock {\em J. Math. Comput. Sci.}, 10(5):1643--1666, 2020.

\bibitem{mallika2017rings}
A.~Mallika and R.~Kala.
\newblock Rings whose cozero-divisor graph has crosscap number at most two.
\newblock {\em Discrete Mathematics, Algorithms and Applications},
  9(06):1750074, 2017.

\bibitem{nikandish2021metric}
R.~Nikandish, M.~Nikmehr, and M.~Bakhtyiari.
\newblock Metric and strong metric dimension in cozero-divisor graphs.
\newblock {\em Mediterranean Journal of Mathematics}, 18(3):1--12, 2021.

\bibitem{patil2021spectrum}
A.~Patil and K.~Shinde.
\newblock Spectrum of the zero-divisor graph of von neumann regular rings.
\newblock {\em Journal of Algebra and its Applications}, 2250193, 2021.

\bibitem{pirzada2021signless}
S.~Pirzada, B.~Rather, R.~U. Shaban, and S.~Merajuddin.
\newblock On signless Laplacian spectrum of the zero divisor graphs of the ring
  $\mathbb{Z}_n$.
\newblock {\em The Korean Journal of Mathematics}, 29(1):13--24, 2021.

\bibitem{pirzada2020distance}
S.~Pirzada, B.~A. Rather, M.~Aijaz, and T.~Chishti.
\newblock On distance signless Laplacian spectrum of graphs and spectrum of
  zero divisor graphs of $\mathbb{Z}_n$.
\newblock {\em Linear and Multilinear Algebra}, pages 1--16, 2020.

\bibitem{pirzada2021normalized}
S.~Pirzada, B.~A. Rather, T.~Chishti, and U.~Samee.
\newblock On normalized Laplacian spectrum of zero divisor graphs of
  commutative ring $\mathbb{Z}_n$.
\newblock {\em Electronic Journal of Graph Theory and Applications (EJGTA)},
  9(2):331--345, 2021.

\bibitem{rather2021laplacian}
B.~A. Rather, S.~Pirzada, T.~A. Naikoo, and Y.~Shang.
\newblock On Laplacian eigenvalues of the zero-divisor graph associated to the
  ring of integers modulo $n$.
\newblock {\em Mathematics}, 9(5):482, 2021.

\bibitem{schwenk1974computing}
A.~J. Schwenk.
\newblock Computing the characteristic polynomial of a graph.
\newblock In {\em Graphs and combinatorics ({P}roc. {C}apital {C}onf.)}, Lecture
  Notes in Math., 406, 153--172, 1974.

\bibitem{westgraph}
D.~B. West.
\newblock {\em Introduction to Graph Theory, 2nd edn.}
\newblock (Prentice Hall), 1996.

\bibitem{wiener1947structural}
H.~Wiener.
\newblock Structural determination of paraffin boiling points.
\newblock {\em Journal of the American chemical society}, 69(1):17--20, 1947.

\bibitem{xu2014survey}
K.~Xu, M.~Liu, K.~C. Das, I.~Gutman, and B.~Furtula.
\newblock A survey on graphs extremal with respect to distance-based
  topological indices.
\newblock {\em MATCH Commun. Math. Comput. Chem}, 71(3):461--508, 2014.

\bibitem{young2015adjacency}
M.~Young.
\newblock Adjacency matrices of zero-divisor graphs of integers modulo $n$.
\newblock {\em Involve, a Journal of Mathematics}, 8(5):753--761, 2015.

\end{thebibliography}
\end{document}